\documentclass[final,nomarks]{dmtcs-episciences}
\usepackage[utf8]{inputenc}
\usepackage{amsthm}
\usepackage{amsmath}
\usepackage{amscd}
\usepackage{t1enc}
\usepackage[mathscr]{eucal}
\usepackage{indentfirst}
\usepackage{pict2e}
\usepackage{epic}
\usepackage[round]{natbib}

\newtheorem{thm}{Theorem}
\newtheorem{lem}[thm]{Lemma}

\newtheorem{prop}[thm]{Proposition}
\newtheorem{conj}[thm]{Conjecture}

\theoremstyle{definition}
\newtheorem{definition}{Definition}[section]

\usepackage{tikz-cd}
\usepackage{pst-node}
\usetikzlibrary{arrows.meta}
\usetikzlibrary{calc}
\usetikzlibrary{patterns}
\renewcommand{\epsilon}{\varepsilon}
\newcommand{\xto}{\xrightarrow}
\newcommand{\vp}{\varphi}
\newcommand{\mbb}{\mathbb}
\newcommand{\Z}{\mbb Z}
\newcommand{\T}{\mathcal T}

\newcommand{\im}{\text{im}}
\newcommand{\Tau}{\mathcal{T}}
\newcommand{\rib}{(G,\rho)}
\newcommand{\ribs}{(G',\rho')}
\newcommand{\ribss}{(G'',\rho'')}

\author{Alex McDonough}
\title{Determining Genus From Sandpile Torsor Algorithms}
\affiliation{
  Department of Mathematics, Brown University
  }
\keywords{sandpile group, sandpile torsor, ribbon graph}
\received{2020-02-29}
\revised{2020-12-18}
\accepted{2020-12-20}

\begin{document}
\publicationdetails{23}{2021}{1}{1}{6176}

\maketitle

\begin{abstract}
    We provide a pair of ribbon graphs that have the same rotor routing and Bernardi sandpile torsors, but different topological genus. This resolves a question posed by M. Chan. We also show that if we are given a graph, but not its ribbon structure, along with the rotor routing sandpile torsors, we are able to determine the ribbon graph's genus. 
\end{abstract}

\section{Introduction}

In this paper, we work with connected graphs that may have multiple edges between the same pair of vertices but we will not allow self loops. We will follow much of the same notation that is used in~\cite{Chan}. For a graph $G$, denote the set of vertices by $V(G)$, the set of edges by $E(G)$, and the set of spanning trees by $\Tau(G)$.\\

\subsection{The Sandpile Group}

For any graph $G$, define the group $\text{Div}(G)$ of $\textit{divisors}$ of $G$ as:

\[\text{Div}(G) := \{\sum_{v \in V(G)}n_vv\mid n_v\in\Z\}.\]

\noindent Define the subgroup $\text{Div}^0(G)$ of \textit{degree-0 divisors} of $G$ as:

$$\text{Div}^0(G) := \{\sum_{v \in V(G)}n_vv\mid n_v\in\Z, \sum_{v \in V(G)}n_v = 0\}$$

\noindent where in general, the \textit{degree} of a divisor is the sum $\sum_{v \in V(G)} n_v$. 

The \textit{Laplacian matrix} $\Delta$ of $G$ is the symmetric matrix defined by:

\[\Delta_{vw} = \begin{cases} -\text{deg}(v) &\text{ if $v = w$}\\
\text{number of edges connecting $v$ to $w$} &\text{ if $v \not= w$}\\
\end{cases}
\]

Finally, define the \textit{sandpile group} or \textit{Picard group} $\text{Pic}^0(G)$ as:

$$\text{Pic}^0(G) := \text{Div}^0(G)/\im(\Delta)$$

We can view the elements of $\text{Div}^0(G)$ as configurations on a graph where we place some number of ``chips'' on each vertex (allowing for negative chips but not fractional chips). The image of the graph Laplacian is generated by ``firing'' and ``unfiring'' vertices of $G$. When a vertex $v$ \textit{fires}, it sends one chip along each edge incident to $v$. This decreases the number of chips at $v$ by the degree of $v$ and increases the number of chips at every other vertex $w$ by the number of edges incident to both $v$ and $w$. When a vertex $v$ \textit{unfires}, it takes in one chip along each edge incident to $v$. This increases the number of chips at $v$ by the degree of $v$ and decreases the number of chips at every other vertex $w$ by the number of edges incident to both $v$ and $w$.

Thus, an equivalent definition of $\text{Pic}^0(G)$ is the abelian group whose elements are configurations of zero total chips on the vertices of $G$, whose binary operation is pointwise addition, and with the equivalence relation given by firing and unfiring vertices. In fact, since unfiring a single vertex is equivalent to firing every other vertex, we can generate our equivalence relation purely by firing vertices. This gives the following useful lemma:

\begin{lem}\label{lem1}
Two elements $S$ and $S'$ of $\text{Div}^0(G)$ are equivalent as elements of $\text{Pic}^0(G)$ if and only if there is a sequence of vertex firings that leads from $S$ to $S'$. 
\end{lem}

\subsection{Sandpile Torsors}

\subsubsection{Relating $\text{Pic}^0(G)$ and $\Tau(G)$}

The narrative of this section is similar to the narrative given in the introduction of~\cite{Chan} and some of these ideas were also explored in~\cite{Wagner}. 

It is a well known fact that the size of the sandpile group of a graph $G$ is the same as the number of spanning trees of $G$ (as shown e.g. in~\cite{Biggs} and~\cite{Holroyd}). Thus, it is natural to ask whether there exists a canonical (automorphism invariant) bijection between these two sets. However, this is impossible in general because there is not always a distinguished spanning tree to associate with the identity element of the sandpile group. For example, a complete graph with more than two vertices has no distinguished spanning tree. 

The next best hope would be if there were a canonical \textit{free transitive action} of $\text{Pic}^0(G)$ acting on $\Tau(G)$. A free transitive action of a group $G$ on a set $S$ is a function $f: G \times S \to S$ such that for any pair $s,s' \in S$, there is a unique $g \in G$ such that $f(g,s) = s'$. A canonical free transitive action is also too much to ask for on a general graph. For example, on a graph with two vertices and three or more edges, each edge is a spanning tree and they are all indistinguishable. Furthermore, even after we select one of the edges, the remaining edges are still indistinguishable. 

To resolve this issue, we introduce additional structure on $G$. For each vertex $v \in V(G)$, assign a cyclic order $\rho_{v}$ to the edges incident to $v$. When this information is provided, $(G, \rho)$ is called a \textit{ribbon graph}, sometimes referred to as a \textit{combinatorial embedding} or a \textit{combinatorial map}. Even with the ribbon graph structure provided, there is not always a canonical choice of free transitive action. For example, if we have a graph with two vertices $v$ and $w$ and three edges $e_1,e_2$ and $e_3$ such that $\rho_{v} = \rho_{w} = (e_1,e_2,e_3)$, then there is no canonical way to decide whether the equivalence class of the sandpile group containing $(v-w)$ or the equivalence class of the sandpile group containing $(w-v)$ should send $e_1$ to $e_2$ (see Figure \ref{vertexneeded}).
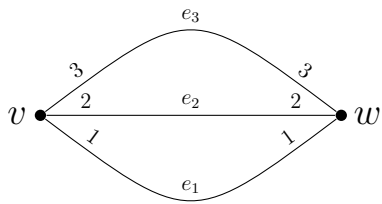
\begin{figure}
\begin{center}
\begin{tikzpicture}
    
    \tikzstyle{every node} = [circle,fill,inner sep=1pt,minimum size = 1.5mm]
    \node[label={west:{\large$v$}}](e) at (3,3) {};
    \node[label={east:{\large$w$}}](f) at (7,3){};

    \tikzstyle{every node} = [draw = none,fill = none,scale = .7]
    \draw (e) .. controls (5,1.5) .. (f)
         node[pos = .1,sloped,above]{1}
         node[pos = .9,sloped,above]{1}
         node[pos = .5,sloped,above]{$e_1$};
    \draw (e) .. controls (5,3) .. (f)
         node[pos = .1,sloped,above]{2}
         node[pos = .9,sloped,above]{2}
         node[pos = .5,sloped,above]{$e_2$};
    \draw (e) .. controls (5,4.5) .. (f)
     	 node[pos = .1,sloped,above]{3}
         node[pos = .9,sloped,above]{3}
         node[pos = .5,sloped,above]{$e_3$};

\end{tikzpicture}
\caption{A ribbon graph with no canonical free transitive action of its sandpile group acting on its spanning trees. The numbers give the cyclic order around each vertex. For the rest of this paper, if no labels are given, the order is assumed to be clockwise.} 
\label{vertexneeded}
\end{center}
\end{figure}

This final ambiguity can be fixed by associating our free transitive action with a distinguished vertex, that we call the $\textit{basepoint}$. 

\begin{definition}A \textit{sandpile torsor} of a ribbon graph $\rib$ is a free transitive action of $\text{Pic}^0(G)$ on $\T(G)$ given a basepoint $v \in V(G)$.
\end{definition}

\theoremstyle{definition}
\begin{definition}A \textit{sandpile torsor algorithm} $\alpha$ is a function whose input is a ribbon graph $(G,\rho)$ and one of its vertices $v \in V(G)$ and whose output is a sandpile torsor on $(G,\rho)$ with basepoint $v$. 
\end{definition}

The two sandpile torsor algorithms we will work with in this paper are the rotor routing process and the Bernardi process. We give a full description of these algorithms in Section~\ref{Torsors}.




\subsection{Summary of Results}

From a ribbon graph $\rib$, we obtain an associated surface by thickening the edges of $G$ and then gluing disks to the boundary components while respecting the cyclic orders given by $\rho$. The \textit{genus} of a ribbon graph $\rib$ is the genus of its associated surface.\footnote{Note that this is not the same as the \textit{combinatorial genus} of $G$ which is defined as $E(G) - V(G) + 1$.} A ribbon graph is called \textit{planar} if its genus is equal to 0. The inspiration for this paper comes from the following theorem proven in~\cite{Chan} for the rotor routing case and~\cite{Baker} for the Bernardi case.

\begin{thm}\label{bigone}
The rotor routing and Bernardi processes  on a  ribbon graph $\rib$ are invariant to the choice of basepoint if and only if $\rib$ is planar. 
\end{thm}

This theorem suggests that we may be able to determine the genus of a ribbon graph from the structure of the sandpile torsors given by a sandpile torsor algorithm, a question posed by Melody Chan~\cite{Chan2}. However, the following theorem shows that this is not the case:

\begin{thm}\label{question1}
Let $\rib$ and $\ribs$ be two ribbon graphs with genera $g$ and $g'$ respectively and let $\alpha_v$ be the rotor routing or Bernardi process with basepoint $v$. Assume that $V(G) = V(G')$, $ \vp: \Tau(G)\to\Tau(G')$ is a bijection, and $\gamma: \text{Pic}^0(G) \to \text{Pic}^0(G')$ is an isomorphism such that for every vertex $v \in V(G)$ the following diagram commutes:

\[ \begin{tikzcd}[column sep = 4cm]
\text{Pic}^0(G) \times \Tau(G) \arrow{r}{\alpha_v(\text{Pic}^0(G))} \arrow[swap]{d}{\gamma\times \vp} & \Tau(G) \arrow{d}{\vp} \\
\text{Pic}^0(G') \times \Tau(G') \arrow{r}{\alpha_v(\text{Pic}^0(G'))}& \Tau(G')
\end{tikzcd}
\]\\

\noindent There exist $\rib$ and $\ribs$ satisfying the above conditions where $g \not= g'$.\\
\end{thm}

We will construct two ribbon graphs demonstrating this theorem in Section \ref{c1}. In fact, we give a very small counterexample where $V(G) = V(G') = 2$ and $E(G) = E(G') = 5$. Note that if we require $g = 0$, Theorem~\ref{question1} does not hold (this is a corollary to Theorem~\ref{bigone}). 

For certain $G$ and $G'$, we can strengthen Theorem~\ref{question1} by requiring $\gamma$ to be a particular kind of map, which we will define in Section~\ref{c2}. 

\begin{thm}\label{question2}
Consider the same conditions as Question \ref{question1}, but where we require $\gamma$ to be induced by the identity on a suitable $V_{gen}\subseteq V(G)$. We can still find $\rib$ and $\ribs$ satisfying the above conditions where $g \not= g'$.
\end{thm}

While we could prove  Theorem~\ref{question1} by restricting to 2-vertex ribbon graphs, we show in Proposition~\ref{2good} that we need more vertices to prove Theorem~\ref{question2}. Nevertheless, we give a family of 3-vertex ribbon graphs that demonstrate this theorem in Section \ref{c2}. 

Because of the failure of these conjectures, any algorithm for determining the genus of a ribbon graph must require more information than just the orbits of the sandpile torsors produced by the rotor routing or Bernardi process. In Section~\ref{example}, we consider the case where we are given $V(G)$ and $E(G)$ but not $\rho$. In this setting, we show that if we are given the map $r_v$ (i.e. the rotor routing torsor with basepoint $v$) for every $v$ , we can determine the genus of $\rib$. Specifically, in Section~\ref{example} we prove:

\begin{thm}\label{algorithm}
Let $\rib$ be a ribbon graph such that $V(G)$ and $E(G)$ are known but $\rho$ is not. Suppose that for every $v \in V(G)$, we are given the map 

$$\text{Pic}^0(G) \times \Tau(G) \xto{r_v(\text{Pic}^0(G))} \Tau(G)$$

\noindent where $r_v$ is the rotor routing torsor with basepoint $v$ (and each $T \in \Tau$ is given as a subset of $E(G)$). Then, it is possible to determine the genus of $\rib$.\\
\end{thm}

\section{Two Sandpile Torsor Algorithms}\label{Torsors}
\subsection{Rotor Routing Process}

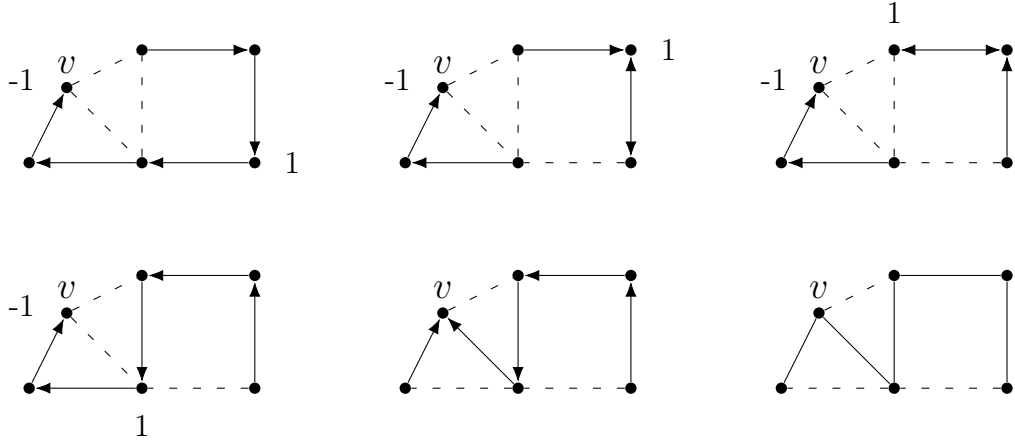
\begin{figure}
\begin{center}
\begin{tikzpicture}
    
    \tikzstyle{every node} = [circle,fill,inner sep=1pt,minimum size = 1.5mm]
    \node (a) at (3,4) {};
    \node(b) at (1.5,4){} ;
    \node[label={north:{\large$v$}}](c) at (2,5){};
    \node(d) at (3,5.5){};
    \node(e) at (4.5,4){};
    \node(f) at (4.5,5.5){};

    \node[fill = none]at (1.4,5.1) {-1};
    \node[fill = none]at (5,4) {1};
    
    \node (a2) at (8,4) {};
    \node(b2) at (6.5,4){} ;
    \node[label={north:{\large$v$}}](c2) at (7,5){};
    \node(d2) at (8,5.5){};
    \node(e2) at (9.5,4){};
    \node(f2) at (9.5,5.5){};
  
   \node[fill = none]at (6.4,5.1) {-1};
   \node[fill = none]at (10,5.5) {1};
    
   \node (a3) at (13,4) {};
    \node(b3) at (11.5,4){} ;
    \node[label={north:{\large$v$}}](c3) at (12,5){};
    \node(d3) at (13,5.5){};
    \node(e3) at (14.5,4){};
    \node(f3) at (14.5,5.5){};
 
    \node[fill = none]at (11.4,5.1) {-1};
    \node[fill = none]at (13,6) {1};
   
    \node (a4) at (3,1) {};
    \node(b4) at (1.5,1){} ;
    \node[label={north:{\large$v$}}](c4) at (2,2){};
    \node(d4) at (3,2.5){};
    \node(e4) at (4.5,1){};
    \node(f4) at (4.5,2.5){};
    
    \node[fill = none]at (1.4,2.1) {-1};
    \node[fill = none]at (3,.5) {1};
    
   \node (a5) at (8,1) {};
    \node(b5) at (6.5,1){} ;
    \node[label={north:{\large$v$}}](c5) at (7,2){};
    \node(d5) at (8,2.5){};
    \node(e5) at (9.5,1){};
    \node(f5) at (9.5,2.5){};
    
   \node (a6) at (13,1) {};
    \node(b6) at (11.5,1){} ;
    \node[label={north:{\large$v$}}](c6) at (12,2){};
    \node(d6) at (13,2.5){};
    \node(e6) at (14.5,1){};
    \node(f6) at (14.5,2.5){};
    \tikzstyle{every node} = [draw = none,fill = none,scale = .7]
    
    \draw [-{Latex[length=2mm,width=1.5mm]}](a) --(b);
    \draw[loosely dashed](c) --(a);
    \draw[loosely dashed] (a) --(d);
    \draw [-{Latex[length=2mm,width=1.5mm]}](e)--(a);
    \draw [-{Latex[length=2mm,width=1.5mm]}](b) --(c);
    \draw[loosely dashed] (c) --(d);
    \draw[-{Latex[length=2mm,width=1.5mm]}](d) --(f);
    \draw[-{Latex[length=2mm,width=1.5mm]}](f) --(e);
    
    \draw [-{Latex[length=2mm,width=1.5mm]}](a2) --(b2);
    \draw[loosely dashed](c2) --(a2);
    \draw[loosely dashed] (a2) --(d2);
    \draw[loosely dashed](e2)--(a2);
    \draw [-{Latex[length=2mm,width=1.5mm]}](b2) --(c2);
    \draw[loosely dashed] (c2) --(d2);
    \draw[-{Latex[length=2mm,width=1.5mm]}](d2) --(f2);
    \draw[{Latex[length=2mm,width=1.5mm]}-{Latex[length=2mm,width=1.5mm]}](f2) --(e2);
    
    \draw [-{Latex[length=2mm,width=1.5mm]}](a3) --(b3);
    \draw[loosely dashed](c3) --(a3);
    \draw[loosely dashed] (a3) --(d3);
    \draw[loosely dashed](e3)--(a3);
    \draw [-{Latex[length=2mm,width=1.5mm]}](b3) --(c3);
    \draw[loosely dashed] (c3) --(d3);
    \draw[{Latex[length=2mm,width=1.5mm]}-{Latex[length=2mm,width=1.5mm]}](d3) --(f3);
    \draw[{Latex[length=2mm,width=1.5mm]}-](f3) --(e3);

    \draw [-{Latex[length=2mm,width=1.5mm]}](a4) --(b4);
    \draw[loosely dashed](c4) --(a4);
    \draw[{Latex[length=2mm,width=1.5mm]}-] (a4) --(d4);
    \draw[loosely dashed](e4)--(a4);
    \draw [-{Latex[length=2mm,width=1.5mm]}](b4) --(c4);
    \draw[loosely dashed] (c4) --(d4);
    \draw[{Latex[length=2mm,width=1.5mm]}-](d4) --(f4);
    \draw[{Latex[length=2mm,width=1.5mm]}-](f4) --(e4);
       
    \draw [loosely dashed](a5) --(b5);
    \draw[{Latex[length=2mm,width=1.5mm]}-](c5) --(a5);
    \draw[{Latex[length=2mm,width=1.5mm]}-] (a5) --(d5);
    \draw[loosely dashed](e5)--(a5);
    \draw [-{Latex[length=2mm,width=1.5mm]}](b5) --(c5);
    \draw[loosely dashed] (c5) --(d5);
    \draw[{Latex[length=2mm,width=1.5mm]}-](d5) --(f5);
    \draw[{Latex[length=2mm,width=1.5mm]}-](f5) --(e5);
    
    \draw [loosely dashed](a6) --(b6);
    \draw(c6) --(a6);
    \draw (a6) --(d6);
    \draw[loosely dashed](e6)--(a6);
    \draw (b6) --(c6);
    \draw[loosely dashed] (c6) --(d6);
    \draw(d6) --(f6);
    \draw(f6) --(e6);
\end{tikzpicture}
\caption{A demonstration of the rotor routing torsor with basepoint $v$ acting on the given spanning tree by the configuration with 1 chip on the bottom right vertex, -1 chips on $v$, and no chips elsewhere.} 
\label{rr}
\end{center}
\end{figure}

The \textit{rotor routing process} is a sandpile torsor algorithm described in~\cite{Holroyd} and based on the ``Eulerian walkers model'' from~\cite{Priezzhev}. 

For $v \in V(G)$, denote $r_v$ as the sandpile torsor with basepoint $v$ determined by the rotor routing process (or the \textit{rotor routing torsor with basepoint $v$} for short). For $S \in \text{Pic}^0(G)$ and $T \in \Tau(G)$, define $r_v(S,T)$ in the following way:

Choose a representative of $S$ with a nonnegative number of chips away from $v$. Then, direct the edges of $T$ so that they point towards $v$ along the path of $T$. There is now one directed edge coming out of every vertex $w \not= v$. This edge is called the \textit{rotor} at $w$. Choose any vertex $w$ that has a positive number of chips. Then, rotate the rotor at $w$ to the next edge in $\rho_w$ and send a chip from $w$ to the other vertex incident to this edge. Continue this process until every vertex has zero chips (at which point the chips have all been deposited at $v$). The resulting position of the rotors is independent of the order that the rotors are rotated and, after removing the directional information, produces a new spanning tree $T'$. See Figure \ref{rr} for an example.

It is proven in~\cite{Holroyd} that $r_v$ is a well-defined free transitive action. 

\subsection{Bernardi Process}

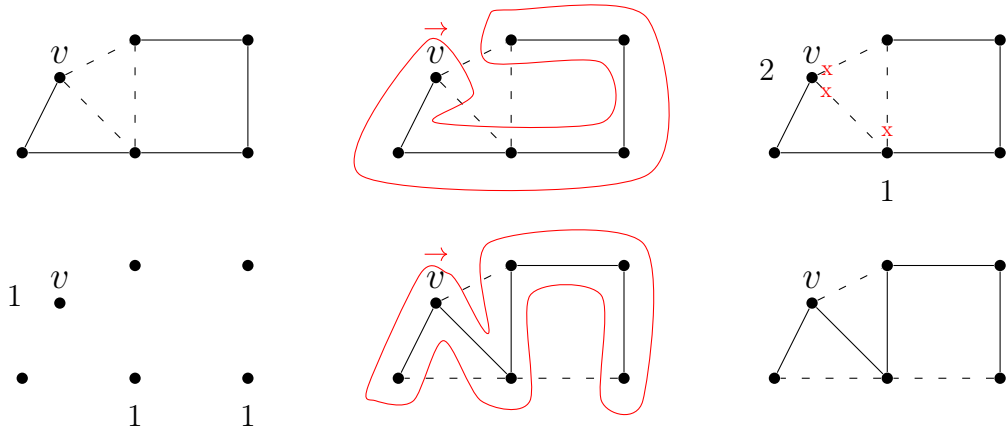
\begin{figure}
\begin{center}
\begin{tikzpicture}
    
    \tikzstyle{every node} = [circle,fill,inner sep=1pt,minimum size = 1.5mm]
    \node (a) at (3,4) {};
    \node(b) at (1.5,4){} ;
    \node[label={north:{\large$v$}}](c) at (2,5){};
    \node(d) at (3,5.5){};
    \node(e) at (4.5,4){};
    \node(f) at (4.5,5.5){};
    
    \node (a2) at (8,4) {};
    \node(b2) at (6.5,4){} ;
    \node[label={north:{\large$v$}}](c2) at (7,5){};
    \node(d2) at (8,5.5){};
    \node(e2) at (9.5,4){};
    \node(f2) at (9.5,5.5){};
    
   \node (a3) at (13,4) {};
    \node(b3) at (11.5,4){} ;
    \node[label={north:{\large$v$}}](c3) at (12,5){};
    \node(d3) at (13,5.5){};
    \node(e3) at (14.5,4){};
    \node(f3) at (14.5,5.5){};
 
    \node[fill = none]at (11.4,5.1) {2};
    \node[fill = none]at (13,3.5) {1};
   
    \node (a4) at (3,1) {};
    \node(b4) at (1.5,1){} ;
    \node[label={north:{\large$v$}}](c4) at (2,2){};
    \node(d4) at (3,2.5){};
    \node(e4) at (4.5,1){};
    \node(f4) at (4.5,2.5){};
    
    \node[fill = none]at (1.4,2.1) {1};
    \node[fill = none]at (3,.5) {1};
    \node[fill = none]at (4.5,.5) {1};
    
   \node (a5) at (8,1) {};
    \node(b5) at (6.5,1){} ;
    \node[label={north:{\large$v$}}](c5) at (7,2){};
    \node(d5) at (8,2.5){};
    \node(e5) at (9.5,1){};
    \node(f5) at (9.5,2.5){};
    
   \node (a6) at (13,1) {};
    \node(b6) at (11.5,1){} ;
    \node[label={north:{\large$v$}}](c6) at (12,2){};
    \node(d6) at (13,2.5){};
    \node(e6) at (14.5,1){};
    \node(f6) at (14.5,2.5){};
    \tikzstyle{every node} = [draw = none,fill = none,scale = .7]
    
    \draw (a) --(b);
    \draw[loosely dashed](c) --(a);
    \draw[loosely dashed] (a) --(d);
    \draw (e)--(a);
    \draw (b) --(c);
    \draw[loosely dashed] (c) --(d);
    \draw(d) --(f);
    \draw(f) --(e);
    
    \draw (a2) --(b2);
    \draw[loosely dashed](c2) --(a2);
    \draw[loosely dashed] (a2) --(d2);
    \draw(e2)--(a2);
    \draw (b2) --(c2);
    \draw[loosely dashed] (c2) --(d2);
    \draw(d2) --(f2);
    \draw(f2) --(e2);
    
    \draw[red] plot [smooth cycle] coordinates {(7.5,4.8)(7,4.4)(9.2,4.4)(9.2,5.2)(7.7,5.2)(7.7,5.8)(8,5.9)(9.8,5.8)(9.8,3.7)(6,3.7)(6.9,5.5)};
    \draw[red,->](6.85,5.65)--(7.15,5.65);
    
    \draw (a3) --(b3);
    \draw[loosely dashed](c3) --(a3)
    node[pos = .15,red]{x};
    \draw[loosely dashed] (a3) --(d3)
    node[pos = .15,red]{x};
    \draw(e3)--(a3);
    \draw (b3) --(c3);
    \draw[loosely dashed] (c3) --(d3)
    node[pos = .15,red]{x};
    \draw(d3) --(f3);
    \draw(f3) --(e3);

    \draw [loosely dashed](a5) --(b5);
    \draw(c5) --(a5);
    \draw (a5) --(d5);
    \draw[loosely dashed](e5)--(a5);
    \draw(b5) --(c5);
    \draw[loosely dashed] (c5) --(d5);
    \draw(d5) --(f5);
    \draw(f5) --(e5);
    
    \draw[red] plot [smooth cycle] coordinates {(7.3,2.25)(7.7,1.6)(7.7,2.8)(9.75,2.8)(9.75,.7)(9.2,.7)(9.2,2.1)(8.25,2.1)(8.25,.7)(7.6,.7)(7.1,1.5)(6.7,.7)(6.2,.7)(6.1,.9)(6.8,2.4)(7.15,2.4)};
    \draw[red,->](6.85,2.65)--(7.15,2.65);
    
    \draw [loosely dashed](a6) --(b6);
    \draw(c6) --(a6);
    \draw (a6) --(d6);
    \draw[loosely dashed](e6)--(a6);
    \draw (b6) --(c6);
    \draw[loosely dashed] (c6) --(d6);
    \draw(d6) --(f6);
    \draw(f6) --(e6);
    
\end{tikzpicture}
\caption{A demonstration of the Bernardi torsor with basepoint $v$ acting on the given spanning tree by the configuration with 1 chip on the bottom right vertex, $-1$ chips on $v$ and no chips elsewhere. Note that it is not a coincidence that this action produces the same spanning tree as the rotor routing action in figure \ref{rr}. It is shown in~\cite{Baker} that the rotor routing and Bernardi actions are identical to each other on planar graphs.} 
\label{b}
\end{center}
\end{figure}

The \textit{Bernardi process} is another sandpile torsor algorithm that is described in~\cite{Baker} based on results from~\cite{Bernardi}.

For $v \in V(G)$, denote $\beta_v$ as the sandpile torsor with basepoint $v$ determined by the Bernardi process (or the \textit{Bernardi torsor with basepoint $v$} for short). For $S \in \text{Pic}^0(G)$ and $T \in \Tau(G)$, define $\beta_v(S,T)$ in the following way:

Consider an edge $e$ incident to vertices $v_1$ and $v_2$ to be composed of two \textit{half-edges} $(e,v_1)$ and $(e,v_2)$. Choose an arbitrary edge $e$ incident to $v$. (The choice of $e$ does not affect the action). We first need to find the \textit{break divisor} associated with each spanning tree.  To get the break divisor associated with $T$, we follow a recursive procedure beginning at the half-edge $(e,v)$ and continuing until we return to $(e,v)$. Informally, this procedure traces around $T$ and places chips the first time it crosses each edge that is not in $T$. Say that our current edge is $(e',v')$. There are 2 cases:

1) If $e' \in T$, we consider the other half edge associated to $e'$, say $(e',w')$. Then, we move to the half edge $(e'',w')$ where $e''$ is the next edge after $e'$ in $\rho_{w'}$ and restart the procedure with $(e'',w')$ as our new half edge.

2) If $e' \not\in T$, we consider the half edge $(\tilde e,v')$ where $\tilde e$ is the next edge after $e'$ in $\rho_{v'}$. Furthermore, if we have not already passed through the other half edge involving $e'$, we place a chip on $v'$. Then we restart the procedure with $(\tilde e, v')$ as our new half edge.\\

This process continues until we return to $(e,v)$. At this point, we will have placed one chip for each edge not in $T$, so this gives us an element of $\text{Div}^g(G)$ for $g = E(G) - V(G) + 1$ where

$$\text{Div}^g(G) := \{\sum_{v \in V(G)}n_vv|n_v\in\Z, \sum_{v \in V(G)}n_v = g\}$$

\noindent It is shown in~\cite{Baker2} that when we apply the Bernardi process to each spanning tree, the resulting chip configurations are all unique as elements of

$$\text{Pic}^g(G) := \text{Div}^g(G)/\im(\Delta).$$

The element of $\text{Pic}^g(G)$ associated to the spanning tree $T$ by this process is called the \textit{break divisor} associated to $T$.\footnote{See ~\cite{Baker} for a complete definition of \textit{break divisor}.} $\beta_v(S,T)$ is given by adding $S$ to the break divisor associated to $T$, which gives us a new element of $\text{Pic}^g(G)$, and then finding the spanning tree $T'$ for which this is the break divisor. See Figure \ref{b} for an example.

It is proven in~\cite{Baker} that $\beta_v$ is a well-defined free transitive action, and an efficient algorithm is provided to find the tree associated with a given break divisor.\\
\section{Counterexamples}\label{counterexample}

We first give an algebraic result that will help to prove both Theorem~\ref{question1} and Theorem~\ref{question2}. In particular, this result says that in order to show that the diagrams commute, we only need to test on a set of generators of $\text{Pic}^0(G)$. 

\begin{lem}\label{algfun}
Let $H$ be a group and $X$ be a set such that $\gamma$ is an automorphism on $H$, $\varphi$ is an automorphism on $X$, and $\alpha$ is a group action from $H \times X \to X$. 

Let $\{h_i\}$ be a set of generators for $H$ and $x$ be an arbitrary element of $X$. If $\varphi(\alpha(h_i,x)) = \alpha(\gamma(h_i),\varphi(x))$ for all $h_i$, then $\varphi(\alpha(h,x)) = \alpha(\gamma(h),\varphi(x))$ for all $h \in H$. 
\end{lem}

\begin{proof}
By definition, we can write any $h \in H$ as $h_1^{k_1}h_2^{k_2}...h_n^{k_n}$. We will proceed by induction on the degree of this monomial. 

When the degree is 1, $h$ is a generator and the result holds automatically. For an arbitrary $h$, assume without loss of generality that $k_1>0$. Then, we can write $h = h_1h'$ where the degree of $h'$ is one less than the degree of $h$. 

For any $x \in X$, the lemma follows from this chain of equalities (which hold by the definition of a group action and the induction hypothesis):

\[\varphi(\alpha(h_1h',x)) = \varphi(\alpha(h_1,\alpha(h',x))) = \alpha(\gamma(h_1),\varphi(\alpha(h',x))) =\] 

\[=\alpha(\gamma(h_1),\alpha(\gamma(h'),\varphi(x)))= \alpha(\gamma(h_1)\gamma(h'),\varphi(x)) = \alpha(\gamma(h),\varphi(x)).\]
\end{proof}
\subsection{Unrestricted $\gamma$ (Theorem~\ref{question1})}\label{c1}

We can prove Theorem~\ref{question1} while only considering ribbon graphs with 2 vertices. For these graphs, each edge is a spanning tree, and there are several other nice properties. We begin with a well known result that is straightforward to prove either by the definition of $\text{Pic}^0(G)$ or by the chip-firing perspective. 

\begin{lem}
If $G$ is a graph with 2 vertices and $n$ edges then $\text{Pic}^0(G) \cong \Z/n\Z$. Furthermore, two configurations are equivalent as elements of $\text{Pic}^0(G)$ if and only if the number of chips on a fixed vertex differ by a multiple of $n$. 
\end{lem}

There is a known formula for the genus of a ribbon graph $\rib$. Define a \textit{cycle} on a ribbon graph $\rib$ as a closed loop such that whenever we enter a vertex, we exit along the next edge in the cyclic order at that vertex. It was shown in~\cite{Edmonds} that these cycles are the faces of the surface associated to $\rib$. Thus, we have the following by Euler's formula (where $cyc(G,\rho)$ is the number of cycles of $\rib$):


\begin{prop}\label{kauf}
For a ribbon graph $\rib$, the genus $g$ satisfies $2g = 2 - |V(G)| + |E(G)| - \text{cyc}\rib$.
\end{prop}

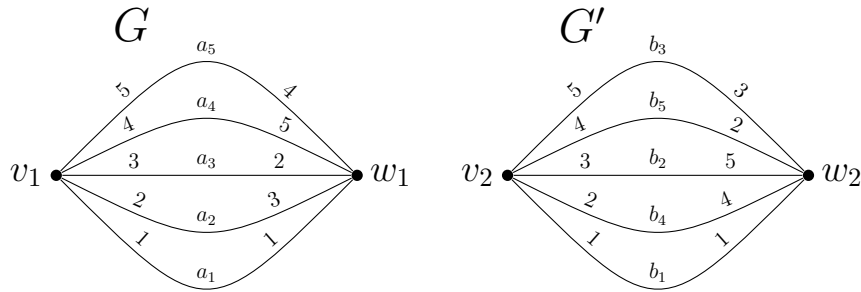
\begin{figure}
\begin{center}
\begin{tikzpicture}
    
    \tikzstyle{every node} = [circle,fill,inner sep=1pt,minimum size = 1.5mm]
    \node[label=west:{\large $v_1$}](b) at (2,3) {};
    \node[label=east:{\large $w_1$}] (c) at (6,3) {};

    \node[label=west:{\large $v_2$}](e) at (8,3){};
    \node[label=east:{\large $w_2$}](f) at (12,3){};

    \tikzstyle{every node} = [draw = none,fill = none, scale = 1.5]
    \node (g) at (3,5) {$G$};
    \node(h) at (9,5) {$G'$};

    \tikzstyle{every node} = [draw = none,fill = none,scale = .7]

    \draw (b) .. controls (4,1) .. (c)
         node[pos = .2,sloped,above]{1}
         node[pos = .5,sloped,above]{$a_1$}
         node[pos = .8,sloped,above]{1};
    \draw (b) .. controls (4,2) .. (c)
         node[pos = .2,sloped,above]{2}
         node[pos = .5,sloped,above]{$a_2$}
         node[pos = .8,sloped,above]{3};
    \draw (b) .. controls (4,3) .. (c)
         node[pos = .2,sloped,above]{3}
         node[pos = .5,sloped,above]{$a_3$}
         node[pos = .8,sloped,above]{2};
    \draw (b) .. controls (4,4) .. (c)
         node[pos = .2,sloped,above]{4}
         node[pos = .5,sloped,above]{$a_4$}
         node[pos = .8,sloped,above]{5};
    \draw (b) .. controls (4,5) .. (c)
         node[pos = .2,sloped,above]{5}
         node[pos = .5,sloped,above]{$a_5$}
         node[pos = .8,sloped,above]{4};

    \draw (e) .. controls (10,1) .. (f)
         node[pos = .2,sloped,above]{1}
         node[pos = .5,sloped,above]{$b_1$}
         node[pos = .8,sloped,above]{1};
    \draw (e) .. controls (10,2) .. (f)
         node[pos = .2,sloped,above]{2}
         node[pos = .5,sloped,above]{$b_4$}
         node[pos = .8,sloped,above]{4};
    \draw (e) .. controls (10,3) .. (f)
         node[pos = .2,sloped,above]{3}
         node[pos = .5,sloped,above]{$b_2$}
         node[pos = .8,sloped,above]{5};
    \draw (e) .. controls (10,4) .. (f)
         node[pos = .2,sloped,above]{4}
         node[pos = .5,sloped,above]{$b_5$}
         node[pos = .8,sloped,above]{2};
    \draw (e) .. controls (10,5) .. (f)
         node[pos = .2,sloped,above]{5}
         node[pos = .5,sloped,above]{$b_3$}
         node[pos = .8,sloped,above]{3};
\end{tikzpicture}
\caption{Two ribbon graphs with the same sandpile torsor structure but different genus.} 
\label{counter1}
\end{center}
\end{figure}

With this formula in mind, we can construct a pair of ribbon graphs that prove Theorem~\ref{question1}. Consider 2 ribbon graphs, $\rib$ and $\ribs$, such that $V(G) = \{v_1,w_1\}$, $V(G') = \{v_2,w_2\}$, and $|E(G)| = |E(G')| = 5$ (see Figure \ref{counter1}). Furthermore, label the edges of $G$ as $a_1$ through $a_5$ such that $\rho_{v_1} = (a_1,a_2,a_3,a_4,a_5)$ and $\rho_{w_1} = (a_1,a_3,a_2,a_5,a_4)$ and label the edges of $G'$ such that $\rho_{v_2} = (b_1,b_4,b_2,b_5,b_3)$ and $\rho_{w_2} = (b_1,b_5,b_3,b_4,b_2)$ (again, see Figure \ref{counter1}). Finally, let $\varphi$ be the map that sends $a_i$ to $b_i$ for each $i$ and $\gamma$ be the map that doubles the number of chips at each vertex. Note that $\gamma$ is an isomorphism because $\text{Pic}^0(G) \cong \Z/5\Z$.  

\begin{prop}\label{prop:counter1}
Let $\rib$, $\ribs$, $\gamma$, and $\vp$ be constructed as above and identify $v_1$ with $v_2$ and $w_1$ with $w_2$. For every vertex $v \in V(G)$ the following diagram commutes, where $\alpha_v$ is the rotor routing process $r$ or the Bernardi process $\beta$ with basepoint $v$:

\[ \begin{tikzcd}[column sep = 4cm]
\text{Pic}^0(G) \times \Tau(G) \arrow{r}{\alpha_v(\text{Pic}^0(G))} \arrow[swap]{d}{\gamma\times \vp} & \Tau(G) \arrow{d}{\vp} \\
\text{Pic}^0(G') \times \Tau(G') \arrow{r}{\alpha_v(\text{Pic}^0(G'))}& \Tau(G')
\end{tikzcd}
\]\\

\noindent However, the genus of $\rib$ is $2$ while the genus of $\ribs$ is $1$.\\
\end{prop}


\begin{proof}

First, we observe that $\beta_{v_i} = r_{w_i}$ and $\beta_{w_i} = r_{v_i}$. To see this, it suffices to show that they match on a generator. 

\begin{align*}\beta_{v_1}(v_1-w_1,\{a_1,a_2,a_3,a_4,a_5\}) = \{a_2,a_3,&a_4,a_5,a_1\} = r_{w_1}(v_1-w_1,\{a_1,a_2,a_3,a_4,a_5\})\\
\beta_{v_2}(v_2-w_2,\{b_1,b_2,b_3,b_4,b_5\}) = \{b_4,b_5,&b_1,b_2,b_3\} = r_{w_2}(v_2-w_2,\{b_1,b_2,b_3,b_4,b_5\})\\
\beta_{w_1}(w_1-v_1,\{a_1,a_2,a_3,a_4,a_5\}) = \{a_3,a_5,&a_2,a_1,a_4\} = r_{v_1}(w_1-v_1,\{a_1,a_2,a_3,a_4,a_5\})\\
\beta_{w_2}(w_2-v_2,\{b_1,b_2,b_3,b_4,b_5\}) = \{b_5,b_1,&b_4,b_2,b_3\} = r_{v_2}(w_2-v_2,\{b_1,b_2,b_3,b_4,b_5\})\\
\end{align*}

Therefore, we only need to prove the result for the rotor routing torsors. Furthermore, by Lemma~\ref{algfun}, it suffices to check a generator of $\text{Pic}^0(G)$ (and we do not have to choose the same generator for $r_v$ as for $r_w$). 

Using $v_i-w_i$ as our generator, we get the following diagram for $r_w$:

\[ \begin{tikzcd}[column sep = 4cm]
(v_i-w_i,\{a_1,a_2,a_3,a_4,a_5\}) \arrow{r}{r_{w_i}(\text{Pic}^0(G))} \arrow[swap]{d}{\gamma\times \vp} & \{a_2,a_3,a_4,a_5,a_1\} \arrow{d}{\vp} \\
(2v_i-2w_i,\{b_1,b_2,b_3,b_4,b_5\}) \arrow{r}{r_{w_i}(\text{Pic}^0(G'))}& \{b_2,b_3,b_4,b_5,b_1\}
\end{tikzcd}
\]\\

Using $w_i-v_i$ as our generator, we get the following diagram for $r_v$:

\[ \begin{tikzcd}[column sep = 4cm]
(w_i-v_i,\{a_1,a_2,a_3,a_4,a_5\}) \arrow{r}{r_{v_i}(\text{Pic}^0(G))} \arrow[swap]{d}{\gamma\times \vp} & \{a_3,a_5,a_2,a_1,a_4\} \arrow{d}{\vp} \\
(2w_i-2v_i,\{b_1,b_2,b_3,b_4,b_5\}) \arrow{r}{r_{v_i}(\text{Pic}^0(G'))}& \{b_3,b_5,b_2,b_1,b_4\}
\end{tikzcd}
\]\\

Finally, we find from direct computation that cyc$\rib = 3$ while cyc$\ribs = 1$. By Proposition \ref{kauf}, this means that the genus of $\rib$ is $1$ while the genus of $\ribs$ is $2$. \end{proof}

\subsection{Restricted $\gamma$ (Theorem~\ref{question2})}\label{c2}
For any $G$ and $G'$ on the same set of vertices, the identity map on $V(G)$ induces a natural isomorphism from $\text{Div}^0(G) \to \text{Div}^0(G')$. However, this isomorphism does not always induce an isomorphism from $\text{Pic}^0(G) \to \text{Pic}^0(G')$ because it is possible that two chip configurations will be firing equivalent on $G$ but not $G'$ (or vice versa). Nevertheless, for certain graphs, we can find natural isomorphisms with respect to appropriate subsets of vertices. 

Let $G$ and $G'$ be two graphs on the same set of vertices. Furthermore, suppose that there is some $V_{gen}\subset V(G)$ satisfying the following properties:
\begin{itemize}
    \item Every element of either $\text{Pic}^0(G)$ and $\text{Pic}^0(G')$ can be written as a linear combination of vertices in $V_{gen}$. In other words, any chip configuration is firing equivalent to one with no chips on vertices outside of $V_{gen}$.
    \item Two chip configurations with no chips outside of $V_{gen}$ are firing equivalent in $G$ if and only if they are firing equivalent on $G'$.
\end{itemize}

Then, let $\hat \gamma$ be a map from $\text{Pic}^0(G) \to \text{Pic}^0(G')$ that we get from the following procedure. Given $S \in \text{Pic}^0(G)$, we first choose a representative for $S$ with no chips outside of $V_{gen}$, which exists by Property 1. Then, we let $\hat \gamma(S)$ be the equivalence class of $\text{Pic}^0(G')$ containing $\text{Id}(S)$ where Id is the map from $\text{Div}^0(G)\to \text{Div}^0(G')$ induced by the identity on $V(G)$. By the second property, we have the following: 

\begin{lem}\label{vgen}
$\hat \gamma$ is a well-defined isomorphism.
\end{lem}

If $\rib$ and $\ribs$ are ribbon graphs on two vertices with the same number of edges, then $G$ and $G'$ must be isomorphic because we do not allow loop edges. This means that we can take $V_{gen}=V(G)$ and this will always give us a isomorphism $\hat \gamma: \text{Pic}^0(G) \to \text{Pic}^0(G')$. Note that $\hat\gamma$ maintains the number of chips on each vertex while the $\gamma$ we used for Proposition~\ref{prop:counter1} doubles them, so this is not sufficent to prove Theorem~\ref{question2}. In fact, we show the following:

\begin{prop}\label{2good}
There are no examples of 2-vertex graphs satisfying Theorem~\ref{question2}.
\end{prop}

\begin{proof}
Let $V(G) = \{v_1,w_2\}$ and $V(G')= \{v_2,w_2\}$ where we identify $v_1$ with $v_2$  and $w_1$ with $w_2$, and also refer to them as $v_i$ and $w_i$ respectively (similarly to the notation used in the proof of Proposition~\ref{prop:counter1}). We can use $r_w$ and $r_v$ (or equivalently $\beta_v$ and $\beta_w$) to determine $\rho_{v_2}$ and $\rho_{w_2}$ in relation to $\rho_{v_1}$ and $\rho_{w_1}$. Then, we show that $cyc(\rho_{v_1}\cdot \rho_{w_1}) = cyc(\rho_{v_2}\cdot \rho_{w_2})$.

Label one of the edges of $G$ as $t_1$ and then for each $k \in [2,n]$, label $r_{w_1}((k-1)v - (k-1)w, t_1)$ as $t_k$. It follows by definition that $\rho_{v_1} = (t_1,....,t_n)$. Then, since $\hat \gamma$ is induced by the identity, for every $k$, we have $r_{w_2}((k-1)v-(k-1)w,\varphi(t_1))$ = $\varphi(t_k)$. Thus, $\rho_{v_2} = (\varphi(t_1)....,\varphi(t_n))$. 

Next, we define $\sigma \in S_n$ to be the permutation such that $\sigma(t_k) = r_{v_1}((k-1)w-(k-1)v,t_1)$. This means that $\rho_{w_1} = (\sigma(t_1),..,\sigma(t_n))$. By the same reasoning as above, it follows that $\rho_{w_2} = (\varphi(\sigma(t_1)),..,\varphi(\sigma(t_n)))$.

Finally, 

$$\rho_{v_2} \cdot \rho_{w_2} = (\varphi(t_1)....,\varphi(t_n))\cdot (\varphi(\sigma(t_1)),..,\varphi(\sigma(t_n))) = \varphi((t_1,...,t_n))\cdot \varphi((\sigma(t_1),..,\sigma(t_n)))=$$

$$= \varphi((t_1,...,t_n)\cdot (\sigma(t_1),..,\sigma(t_n))) = \varphi(\rho_{v_1}\cdot \rho_{w_1}).$$

$\varphi$ is a bijection, so it does not affect the number of cycles in the resulting product. This means that $\rib$ and $\ribs$ must have the same genus.
\end{proof}

\begin{figure}
\begin{center}
\begin{tikzpicture}
    
    \tikzstyle{every node} =  [circle,fill,inner sep=1pt,minimum size = 1.5mm]
    \node[label={north west:{\large $v_1$}}] (a) at (0,5) {};
    \node[label={north:{\large $z_1$}}] (b) at (3,5) {};
    \node[label={north east:{\large $w_1$}}] (c) at (6,5) {};

    \node[label={north west:{\large $v_2$}}](d) at (9,5){};
    \node[label={north:{\large $z_2$}}](e) at (11,5){};
    \node[label={north east:{\large $w_2$}}](f) at (15,5){};

    \tikzstyle{every node} = [draw = none,fill = none, scale = 1.5]
    \node (g) at (2,6.5) {$G$};
    \node(h) at (11,6.5) {$G'$};

    \tikzstyle{every node} = [draw = none,fill = none,scale = .7]
    \draw (a) .. controls(1.5,6).. (b)
         node[pos = .2,sloped,above]{1}
         node[pos = .8,sloped,above]{x + 1};
    \draw (a) .. controls (1.5,4) .. (b)
         node[pos = .2,sloped,below]{2}
         node[pos = .8,sloped,below]{x + 2};
    \draw (b) .. controls (4.5,3) .. (c)
         node[pos = .2,sloped,below]{1}
         node[pos = .8,sloped,below]{1};
    \draw (b) .. controls (4.5,4) .. (c)
         node[pos = .2,sloped,below]{2}
         node[pos = .8,sloped,below]{2};
    \draw[loosely dashed] (b) .. controls (4.5,5) .. (c);
    \draw (b) .. controls (4.5,6) .. (c)
         node[pos = .2,sloped,above]{x-1}
         node[pos = .8,sloped,above]{x-1};
    \draw (b) .. controls (4.5,7) .. (c)
         node[pos = .2,sloped,above]{x}
         node[pos = .8,sloped,above]{x};

    \draw (d) .. controls (10,5) .. (e)
         node[pos = .15,sloped,below]{1}
         node[pos = .8,sloped,below]{2x + 1};
    \draw (e) .. controls (13,2.5) .. (f)
         node[pos = .2,sloped,below]{1}
         node[pos = .8,sloped,below]{1};
    \draw (e) .. controls (13,3.5) .. (f)
         node[pos = .2,sloped,below]{2}
         node[pos = .8,sloped,below]{2};
    \draw[loosely dashed] (e) .. controls (13,4.5) .. (f);
    \draw[loosely dashed] (e) .. controls (13,5.5) .. (f);
    \draw (e) .. controls (13,6.5) .. (f)
         node[pos = .2,sloped,above]{2x-1}
         node[pos = .8,sloped,above]{2x-1};
    \draw (e) .. controls (13,7.5) .. (f)
         node[pos = .2,sloped,above]{2x}
         node[pos = .8,sloped,above]{2x};
\end{tikzpicture}
\caption{Two ribbon graphs with the same rotor routing/ Bernardi torsors but different genus} 
\label{twographs}
\end{center}
\end{figure}
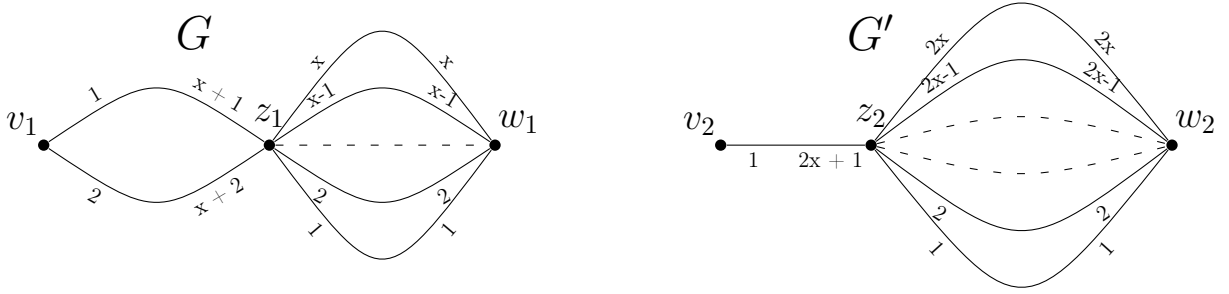

Proposition~\ref{2good} says that in order to prove Theorem~\ref{question2}, we will need to work with ribbon graphs that have at least 3 vertices. Let $x$ be any odd integer. Consider two ribbon graphs, $\rib$ and $\ribs$ such that $|V(G)| = |V(G')| = 3$. Call the elements of $V(G)$ $v_1$, $z_1$, and $w_1$, and call the elements of $V(G')$ $v_2$, $z_2$, and $w_2$. Connect $v_1$ and $z_1$ with 2 edges, $z_1$ and $w_1$ with $x$ edges, $v_2$ and $z_2$ with 1 edge, and $z_2$ and $w_2$ with $2x$ edges (see Figure \ref{twographs}). For the cyclic ordering $\rho_{z_1}$, set the 2 edges that connect to $v_1$ to be next to each other. Furthermore, set the cyclic order of edges connecting $z_1$ to $w_1$ to be the same for $\rho_{z_1}$ as $\rho_{w_1}$, and likewise, set the cyclic ordering of edges connecting $z_2$ to $w_2$ to be the same for $\rho'_{z_2}$ as $\rho'_{w_2}$ (again see Figure \ref{twographs}).  

\begin{thm}\label{ugly}
For any $g\in \Z_{>0}$, let $\rib$ and $\ribs$ be constructed as above with $x = 2g+1$. If we identify the vertices of $G$ with the vertices of $G'$, then $\text{Pic}^0(G) \cong \text{Pic}^0(G')$ and $\{v_i,w_i\}$ the $V_{gen}$ requirements of Lemma~\ref{vgen}. Furthermore, the diagram in Theorem~\ref{question2} commutes. However, the genus of $\rib$ is $g$ while the genus of $\ribs$ is $2g$.\\
\end{thm}

Note that if $x$ is even, then this theorem does not hold. In particular, $\text{Pic}^0(G)\cong \Z/2\Z \oplus \Z/x\Z$ and $\text{Pic}^0(G') \cong \Z/2x\Z$. These two groups are only isomorphic if $x$ is odd. 

Before we prove the theorem, we prove a lemma which gives a sufficient condition for two specific basepoints to be equivalent with respect to either of our sandpile torsor algorithms. In particular, when we apply this lemma to the ribbon graphs in Figure~\ref{twographs}, we find that the rotor routing process is the same with basepoint $v_i$ as with $z_i$ (where $i$ is either 1 or 2) and the same is true for the Bernardi process.

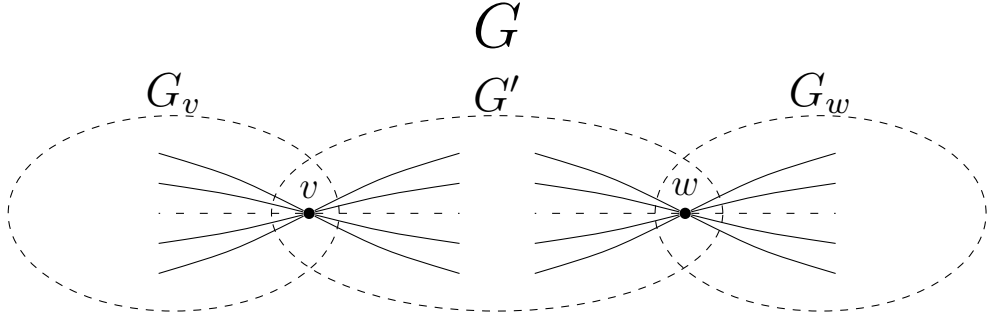
\begin{figure}
\begin{center}
\begin{tikzpicture}
    
    \tikzstyle{every node} = [circle,fill,inner sep=1pt,minimum size = 1.5mm]
    \node[label= north:{\large $v$ }] (a) at (4,4) {};
    \node[label= north:{\large $w$ }](b) at (9,4) {};

    \tikzstyle{every node} = [draw = none,fill = none, scale = 1.5]
    \node (c) at (2.2,5.6) {$G_v$};
    \node(d) at (6.5,5.6) {$G'$};
    \node(e) at (10.8,5.6) {$G_w$};
    \tikzstyle{every node} = [draw = none,fill = none, scale = 2]
    \node(f) at (6.5,6.5) {$G$};

    \tikzstyle{every node} = [draw = none,fill = none,scale = .7]
    \draw (a)  .. controls (3,3.5) .. (2,3.2);
    \draw (a)  .. controls (3,3.75) .. (2,3.6);
    \draw [loosely dashed](a) .. controls (3,4) .. (2,4);
    \draw (a) .. controls (3,4.25) .. (2,4.4);
    \draw (a) .. controls (3,4.5) .. (2,4.8);

    \draw[dashed] (2.2,4) ellipse (2.2 and 1.3);

    \draw (a)  .. controls (5,3.5) .. (6,3.2);
    \draw (a)  .. controls (5,3.75) .. (6,3.6);
    \draw [loosely dashed](a) .. controls (5,4) .. (6,4);
    \draw (a) .. controls (5,4.25) .. (6,4.4);
    \draw (a) .. controls (5,4.5) .. (6,4.8);

    \draw[dashed] (6.5,4) ellipse (3 and 1.3);

    \draw (b)  .. controls (8,3.5) .. (7,3.2);
    \draw (b)  .. controls (8,3.75) .. (7,3.6);
    \draw [loosely dashed](b) .. controls (8,4) .. (7,4);
    \draw (b) .. controls (8,4.25) .. (7,4.4);
    \draw (b) .. controls (8,4.5) .. (7,4.8);

    \draw[dashed] (10.8,4) ellipse (2.2 and 1.3);

    \draw (b)  .. controls (10,3.5) .. (11,3.2);
    \draw (b)  .. controls (10,3.75) .. (11,3.6);
    \draw [loosely dashed](b) .. controls (10,4) .. (11,4);
    \draw (b) .. controls (10,4.25) .. (11,4.4);
    \draw (b) .. controls (10,4.5) .. (11,4.8);

\end{tikzpicture}
\caption{The Rotor Routing/ Bernardi torsors at $v$ and $w$ are the same on $\rib$ when $\ribs$ is planar} \label{planarish}
\end{center}

\end{figure}

Let $\rib$ be a ribbon graph and let $v,w \in V(G)$. We split $G$ into 3 subgraphs labeled $G'$, $G_v$, and $G_w$ using the following construction (which is given in Figure \ref{planarish}): 

For any $e \in E(G)$,  
\begin{itemize}
    \item If every path from $e$ to $w$ passes through $v$, $e \in G_v$.
    \item If every path from $e$ to $v$ passes through $w$, $e \in G_w$.
    \item If neither above condition is met, $e \in G'$. 
\end{itemize}

When including an edge in any of these subgraphs, we also include both incident vertices. Furthermore, we always require $v \in G_v$ and $w \in G_w$, even when $G_v$ or $G_w$ contains no edges. It is immediate that $G_v \cap G' = \{v\}$ and $G_w \cap G' = \{w\}$. Let $\ribs$ be the restriction of $\rib$ to $G'$.

\begin{lem}
\label{useful}
For a ribbon graph $\rib$ with $v,w \in V(G)$, construct $\ribs$ as above. Let the following conditions hold:
 \begin{itemize}
     \item $\ribs$ is planar. 
     \item The edges of $G'$ that are incident to $v$ are all sequential in $\rho_v$.
     \item The edges of $G'$ that are incident to $w$ are all sequential in $\rho_w$.
 \end{itemize}
Then, $\alpha_v$ and $\alpha_w$ are equivalent sandpile torsors when $\alpha$ is replaced by either $r$ or $\beta$. 
\end{lem}

\begin{proof} 
First, we consider the case where $\alpha$ is the rotor routing process. Let $S$ be any element of $\text{Div}^0(G)$ that has a nonnegative number of chips away from $v$. Let $S'$ be an element of $\text{Div}^0(G)$ that is equivalent to $S$ as an element of $\text{Pic}^0(G)$ and has a nonnegative number of chips away from $w$. 

We need to show that for any spanning tree $T \in \Tau(G)$, we have $r_v(S,T) = r_w(S',T)$. We can begin our evaluation of each of these rotor routing torsors by performing rotor routing on $G_v$ and $G_w$ until all vertices in $G \setminus G'$ have no chips on them. Because $S$ and $S'$ are in the same sandpile equivalence class, and because the rotors in $G_v$ and $G_w$ will always point towards $v$ and $w$ respectively, the resulting portions of the spanning tree outside of $G'$ is the same with either basepoint. Furthermore, if a chip ever leaves $G'$ during rotor routing (say WLOG that it enters $G_v$), then this happens because the rotor at $v$ rotates into $G_v$. For the chip to return to $G'$ (which must happen eventually), the rotor at $v$ must keep spinning until it returns to an edge in $G'$. This drops one chip across each edge in $G_v$ incident to $v$. The effect of this rotation is the same as firing $v$ in the subgraph $G_v$ which has no effect on the resulting tree. By Theorem \ref{bigone}, we know that $r_v = r_w$ when we restrict to $\ribs$ and the above analysis shows that this is also true on $\rib$.

The Bernardi action is even simpler. If we start each tour with the first edge in $G'$ connected to the basepoint vertex, then the tours will go around $G_v$ and $G_w$ in the same direction. Thus, the effect of these subgraphs on the break divisors will be the same for each basepoint vertex. Because the two Bernardi actions are the same on $G'$ and we alter each of them in the same way, they are also the same on $G$.\end{proof}

Now we are ready to prove the Theorem \ref{ugly}.

\begin{proof}[Proof of Theorem \ref{ugly}] For each ribbon graph, there are $2x$ spanning trees, which means that this is also the size of the sandpile groups. We claim that the sandpile element $v_i - w_i$ has order $2x$ in both $\text{Pic}^0(G)$ and $\text{Pic}^0(G')$. This means that it must generate the sandpile group for both graphs. Furthermore, since there are no chips on $z_i$, the pair $\{v_i,w_i\}$ satisfies the $V_{gen}$ requirements of Lemma~\ref{vgen}. 

Label the spanning trees of $G_1$ as $[a,b]$ where $a$ is the index of the edge between $v_1$ and $z_1$ (either 1 or 2) and $b$ is the index of the edge between $z_1$ and $w_1$ in cyclic order (ranging from 1 to $x$). Label the spanning trees of $G_2$ as $[a]$ with $a$ the index of the edge between $z_2$ and $w_2$ (ranging from 1 to $2x$). Our claim follows if we show that 

\begin{equation}\label{trees1}
\{[1,1], r_{w_1}(1,0,-1)[1,1],  r_{w_1}(2,0,-2)[1,1],...,r_{w_1}(2x-1,0,1-2x)[1,1]\}
\end{equation}

\noindent are all distinct spanning trees of $G_1$ and

\begin{equation}\label{trees2}
\{[1], r_{w_2}(1,0,-1)[1],  r_{w_2}(2,0,-2)[1],...,r_{w_2}(2x-1,0,1-2x)[1]\}\end{equation}

\noindent are all distinct spanning trees of $G_2$ (where we could replace $r_{w_i}$ with any other sandpile torsor). 

On $G_1$, $r_{w_1}(1,0,-1)$ switches the edge between $v_1$ and $z_1$ and then shifts the edge between $z_1$ and $w_1$ up by 1. The only special case is when we get to the last edge between $z_1$ and $w_1$ and shift over to the edges between $v_1$ and $z_1$. However, this just causes the edge between $v_1$ and $z_1$ to shift twice which does not change it and then we get the first edge between $z_1$ and $w_1$ before depositing the chip at $w_1$. Thus, the trees given in~\ref{trees1} are:

$$\{[1,1], [2,2], [1,3], ..., [1,x], [2,1], ..., [2,x] \}.$$

On $G_2$, $r_{w_2}(1,0,-1)$ simply switches to the next edge between $z_2$ and $w_2$. Thus, the trees given in~\ref{trees2} are:

$$\{[1], [2], [3], ..., [2x-1], [2x] \}.$$

In both cases, we get $2x$ distinct trees. Additionally, this result, along with Lemma~\ref{algfun}, tells us that there is a unique bijection $\varphi:\Tau(G_1) \to \Tau(G_2)$ that will make the diagram from Theorem~\ref{question2} commute when our sandpile torsor is $r_{w_i}$. In particular, we let $\vp([a,b])= [b]$ when $a$ and $b$ have the same parity, and $\vp([a,b]) =[b+x]$ when $a$ and $b$ are of opposite parity. 

We now need to check that this same bijection will cause the diagram to commute when we replace $r_{w_i}$ with $r_{v_i}, \beta_{v_i}$, or $\beta_{w_i}$ (and by Lemma~\ref{algfun}, we only need to check on a generator).

By similar computation to above, we find that

$$\{[1,1], r_{v_1}(-1,0,1)[1,1],  r_{v_1}(-2,0,2)[1,1],...,r_{v_1}(1-2x,0,2x-1)[1,1]\}$$

\noindent is equal to

$$\{[1,1], [2,2], [1,3], ..., [1,x], [2,1], ..., [2,x] \}$$

\noindent while 

$$\{[1], r_{v_2}(-1,0,1)[1],  r_{v_2}(-2,0,2)[1],...,r_{v_2}(1-2x,0,2x-1)[1]\}$$

\noindent is equal to 

$$\{[1], [2], [3], ..., [2x-1], [2x] \}.$$

These trees occur in the same order that they did for $r_{w_i}$, so the same bijection holds.

Now, we look at the Bernardi torsors. On $G_1$, consider $\beta_{v_1}(1,0,-1)$. We will start the Bernardi tour on the first edge connecting $v_1$ to $z_1$. If this edge is part of our spanning tree, we will place one chip on $z_1$ when the tour reaches the other edge between $v_1$ and $z_1$. Otherwise, we place one chip at $v_1$ at the very beginning. Additionally, we place one chip on $z_1$ for each edge between $z_1$ and $w_1$ before the edge of our spanning tree, and one chip on $w_1$ for each edge between $z_1$ and $w_1$ after the edge of our spanning tree. Thus there are 2 cases:

If the spanning tree is $[1,k]$ for some $k$, then the break divisor is $(0,k,x-k)$. If the spanning tree is $[2,k]$ for some $k$, then the break divisor is $(1,k-1,x-k)$. In the first case, adding $(1,0,-1)$ gives $(1,k,x-k-1)$ which is the break divisor for $[2,k+1]$ (if $k = x$, we have the divisor $(1,x,-1)$ which is equal to $(1,0,x-1)$ after unfiring $w_1$ once. This is the break divisor for $[2,1]$). In the second case, adding $(1,0,-1)$ gives $(2,k-1,x-k-1)$. After firing $v_1$ once, we get $(0,k+1,x-k-1)$ which is the break divisor for $[1,k+1]$ (if $k = x$, we have the divisor $(0,x+1,-1)$ which is equal to $(0,1,x-1)$ after unfiring $w_1$ once. This is the break divisor for $[1,1]$.) This means that

$$\{[1,1], \beta_{v_1}(1,0,-1)[1,1],  \beta_{v_1}(2,0,-2)[1,1],...,\beta_{v_1}(2x-1,0,1-2x)[1,1]\}$$

\noindent is equal to

$$([1,1], [2,2], [1,3], ..., [1,x], [2,1], ..., [2,x] )$$

\noindent which is the same as the $r_{w_1}$ action. \\

The case for $\beta_{w_1}(-1,0,1)$ is completely similar and yields that

$$\{[1,1], \beta_{w_1}(-1,0,1)[1,1],  \beta_{w_1}(-2,0,2)[1,1],...,\beta_{w_1}(1-2x,0,2x-1)[1,1]\}$$

\noindent is equal to

$$\{[1,1], [2,2], [1,3], ..., [1,x], [2,1], ..., [2,x] \}$$

\noindent which is the same as $r_{v_1}$.\\

On $G_2$, because the edge between $v_2$ and $z_2$ is in every spanning tree, we can ignore it and look at the other two vertices.  On a two vertex graph, the rotor routing process at one basepoint produces the same tensor as the Bernardi process at the other basepoint. Thus, $\beta_{v_2} = r_{w_2}$ and $\beta_{w_2} = r_{v_2}$. Combined with our previous results that $\beta_{v_1} = r_{w_1}$ and $\beta_{w_1} = r_{v_1}$, we conclude that $\beta_{v_i} = r_{w_i}$ and $\beta_{w_i} = r_{v_i}$. This, the diagram commutes for either sandpile torsor algorithm.\\

The only thing left to show is that the genus of $G_1$ is $g$ while the genus of $G_2$ is $2g$. This is a direct application of Lemma~\ref{kauf}.\end{proof}

\section{Genus From Rotor Routing when the Graph is Known}\label{example}

In order to determine the genus of a ribbon graph, we need more information than just the rotor routing or Bernardi torsors. For this final section of the paper, we work with a ribbon graph $\rib$ for which $G$ is known, but $\rho$ is not. This alone is not enough to determine the genus of $\rib$, but we show that if we are also given the rotor routing action at each basepoint, we can calculate the genus. In other words, we prove Theorem~\ref{algorithm}. 

Our method of proof is to take an arbitrary vertex of our ribbon graph and show that the cyclic order of edges around it is essentially uniquely determined. Then, we can apply Proposition \ref{kauf} to determine the ribbon graph's genus.

\theoremstyle{definition}
\begin{definition}
Let $\rib$ be a ribbon graph and $v \in V(G)$. A \textit{$v$-component} of $\rib$ is the full ribbon subgraph induced on the vertices of a connected component of $G \setminus v$ union $\{v\}$.\end{definition}

Note that $\rib$ has multiple $v$-components if and only if $v$ is a cut vertex. Furthermore, the intersection of any two $v$-components is $v$. In Figure \ref{bigpic}, the lower ribbon graph is a $v$-component of the upper ribbon graph.

\begin{lem}\label{tree}
Let $\rib$ be a ribbon graph with a vertex $v$. Let $e_1$ and $e_2$ be two edges incident to $v$ in the same $v$-component, and let $w_1$ and $w_2$ be their other incident vertices respectively. There exists a spanning tree $T$ of $\rib$ such that: 
\begin{itemize}
    \item $e_1 \in T$,
    
    \item $e_2 \not\in T$, and
    
    \item the path from $w_2$ to $v$ using edges in $T$ passes through $w_1$.
\end{itemize}
\end{lem}

\begin{proof} By the definition of $v$-components, there is a path between $w_1$ and $w_2$ that does not pass through $v$. Because this path does not pass through $v$, adding $e_1$ to the path will not give us a cycle. Then, we expand to any spanning tree. This spanning tree must not contain $e_2$ or we would have a cycle, so all three conditions are met. \end{proof}

Consider $\rib$, $v$, $e_1$, $e_2$, $w_1$, and $w_2$ as given in Lemma~\ref{tree}. Let $T$ be a spanning tree satisfying the conditions of this lemma and $\ribs$ be the $v$-component containing $e_1$ and $e_2$. Let $T$ be a spanning tree satisfying the conditions of Lemma \ref{tree}, and let $T'$ be the restriction of $T$ to $G'$ (which  is a spanning tree of $G'$). 

Let $S \in \text{Div}^0(G)$ be the configuration that places 1 chip on $v$, $-1$ chips on $w_2$, and 0 chips elsewhere. Let $r_{w_2}$ be the rotor routing torsor on $\rib$ with basepoint $w_2$. Let $\hat T = r_{w_2}(S,T)$ and $\hat T'$ be the restriction of $\hat T$ to $E(G')$.

\begin{prop}\label{same}
Consider the construction above. The edge $e_2$ is directly after $e_1$ in $\rho'_v$ if and only if $\hat T' = T' \cup e_2 \setminus e_1$.
\end{prop}

\begin{proof} In the evaluation of $r_{w_2}(S, T)$, the single chip on $v$ travels around the graph until it reaches $w_2$. Whenever the chip enters a $v$-component other than $\ribs$, say $\ribss$, it remains in $\ribss$ until it returns to $v$. While the chip is in $\ribss$, it can only shift edges within $\ribss$. In particular, it will not affect $\hat T'$. After the chip has returned to $v$, it will move on to the next edge in the cyclic order around $v$, and the effect on $\hat T'$ will be the same as if the rotor had spun an extra time without sending the chip. Hence, it suffices to consider the case where $G$ has only one $v$-component.

After this simplification, the forward direction of the proof is immediate because if $e_2$ is the next edge after $e_1$ in the cyclic order around $v$, the rotor routing torsor will have a single step which exchanges $e_1$ for $e_2$ and then deposits the chip to $w_2$. The result is our desired tree.

For the other direction, we proceed by contradiction. Assume that the edges $a_1,..,a_k$ all fall between $e_1$ and $e_2$ in the cyclic order around $v$. Consider the configuration $S' \in \text{Div}^0(G)$ that places $k+1$ chips on $v$ and $-d_x$ chips on each other vertex $x$ where $d_x$ is the number of edges in $\{a_1,..,a_k,e_2\}$ that are incident to $x$. Then, the evaluation of $r_{w_2}(S', T)$ rotates the rotor at $v$ around $k+1$ times so that it is now at $e_2$. Thus, the resulting tree is $T' \cup e_2 \setminus e_1$. To establish our contradiction, we need to show that $r_{w_2}(S', T) \not= r_{w_2}(S, T)$. Because the rotor routing action is free and transitive, this statement reduces to showing that $S$ and $S'$ are not equivalent as elements of $\text{Pic}^0(G)$, which is the same as showing that $S - S'$ is not equivalent to the identity.

The configuration $S - S'$ has $-k$ chips on $v$ and $d_x$ chips on each other vertex $x$, where $d_x$ is the number of edges in $\{a_1,..,a_k\}$ that are incident to $x$. By Lemma \ref{lem1}, if $S-S'$ is equivalent to the identity, then we can get from this configuration to the configuration where there are no chips on the graph merely by firing vertices. Because firing a vertex is the only way to decrease the number of chips it has, every vertex that begins with chips must be fired. Additionally, any non-$v$ vertex adjacent to a fired vertex must be fired because it begins with no chips and gains a chip once the adjacent vertex has been fired. By recursion, this means that any vertex that is connected to a fired vertex by a path not passing through $v$ must fire. We assumed that every vertex is on the same $v$-component, so every non-$v$ vertex must fire. Additionally, since every edge in $E(G)$ is incident to a non-$v$ vertex, every edge must have a chip travel across it. Since there are at least $k+2$ edges incident to $v$, $v$ will eventually have a positive number of chips and must also fire. However, firing every vertex is equivalent to firing no vertices, meaning $S-S'$ must be the configuration where there are no chips. This is a contradiction because we assumed that there were edges between $e_1$ and $e_2$ \end{proof}

This proposition implies that on any cut-free ribbon graph $\rib$, given the necessary inputs for Theorem \ref{algorithm}, we can precisely calculate $\rho_{v_k}$ and thus, by Proposition \ref{kauf}, also the genus of $\rib$. However, knowing the restriction of $\rho_v$ to each $v$-component is not generally enough information to determine genus. We will also need information about when edges from one $v$-component fall between edges of a second $v$-component. This is the content of the next two lemmas.\\

Let $\rib$ be a ribbon graph with a vertex $v$. Let $e_1$ and $e_2$ be two sequential edges within a $v$-component, and $w_1$ and $w_2$ be their other incident vertices respectively. Consider a different $v$-component $\ribs$ such that $a_1,...,a_k$ are the edges in $E(G')$ that are between $e_1$ and $e_2$ in $\rho_v$. Let $T$ be a spanning tree satisfying the conditions of Lemma \ref{tree}, and $T'$ be the restriction of $T$ to $E(G')$. 

Let $S \in \text{Div}^0(G)$ be the configuration with $1$ chip on $v$, $-1$ chips on $w_2$, and 0 chips elsewhere. Additionally, let $r_{w_2}$ be the rotor routing action on $\rib$ with basepoint $w_2$, $\hat T = r_{w_2}(S, T)$, and $\hat T'$ be the restriction of $\hat T$ to $E(G')$. 

We compare the tree $\hat T'$ to a tree we obtain by restricting to $\ribs$ from the start. Let $S' \in \text{Div}^0(G')$ be the configuration with $-k$ chips on $v$ and $d_x$ chips on each other vertex $x\in V(G')$, where $d_x$ is the number of edges incident to $x$ in $\{a_1,...,a_k\}$. Finally, let $r'_{v}$ be the rotor routing torsor on $\ribs$ with basepoint $v$.

\begin{lem}\label{dif1}
In the construction above, $r'_{v}(S', T') = \hat{T'}$.
\end{lem}

\begin{proof} As the rotor at $v$ rotates from $e_1$ to $e_2$, it will pass through each of the edges $\{a_1,..,a_k\}$ once. By the same reasoning as discussed in the previous proof, any edges not in $E(G')$ that the rotor passes through will have no effect on $\hat T'$. Every time the rotor reaches edge $a_i$, one chip is transferred from $v$ to the other vertex incident to $a_i$ (call this vertex $b_i$). Then, the chip travels around in $\ribs$ until it returns to $v$. This has the same effect on the rotors in $\ribs$ as if we placed a single chip on $b_i$ and evaluated $r'_v$. Combining these single chip addition configurations gives us the element of the sandpile group $S'$. See Figure \ref{bigpic} \end{proof} 
\begin{figure}
\begin{center}
\begin{tikzpicture}[scale = 0.8]
    
    \tikzstyle{every node} = [circle,fill,inner sep=1pt,minimum size = 1.5mm]
    \node [label = {north east:{\large$v$}}](a) at (4,4) {};
    \node[label = {north:$w_1$}](b) at (2.5,4){};
    \node(c) at (3,5){};
    \node(d) at (4,5.5){};
    \node(e) at (5,5){};
    \node[label = {north:$w_2$}](f) at (5.5,4){};
    \node(g) at (5,3){};
    \node(h) at (4,2.5){};
    \node(i) at (3,3){};

   \node(j) at (5,2){};
   \node(k) at (2.5,2){};
   \node(l) at (7,3){};

   \node(m) at (4,6.5){};
   \node(n) at (7,5){};

    \node[label = {north east:{\large$v$}}](a2) at (12,4) {};
    \node(b2) at (10.5,4){} ;
    \node (c2) at (11,5){};
    \node(d2) at (12,5.5){};
    \node(e2) at (13,5){};
    \node(f2) at (13.5,4){};
    \node(g2) at (13,3){};
    \node(h2) at (12,2.5){};
    \node(i2) at (11,3){};

   \node(j2) at (13,2){};
   \node(k2) at (10.5,2){};
   \node(l2) at (15,3){};

   \node(m2) at (12,6.5){};
   \node(n2) at (15,5){};
    \node[label = {west:{\large$v$}}](a3) at (4,-1) {};
    \node(d3) at (4,.5){};
    \node(g3) at (5,-2){};
    \node(h3) at (4,-2.5){};
    \node[label = {west:{\large$v$}}](a4) at (12,-1) {};
    \node(d4) at (12,.5){};
    \node(g4) at (13,-2){};
    \node(h4) at (12,-2.5){};

    \tikzstyle{every node} = [draw = none,fill = none,scale = .7]
    \draw [-{Latex[length=2mm,width=1.5mm]}](a) --(b);
    \draw[-{Latex[length=2mm,width=1.5mm]}] (c) --(a);
    \draw[loosely dashed] (a) --(d);
    \draw[loosely dashed] (a)--(e);
    \draw[loosely dashed]  (a) --(f);
    \draw [-{Latex[length=2mm,width=1.5mm]}](g) --(a);
    \draw[loosely dashed] (a) --(h);
    \draw[-{Latex[length=2mm,width=1.5mm]}](i) --(a);

    \draw (b)[loosely dashed]  --(i);
    \draw [-{Latex[length=2mm,width=1.5mm]}](k)--(j);
    \draw [-{Latex[length=2mm,width=1.5mm]}](b)--(k);
    \draw [-{Latex[length=2mm,width=1.5mm]}](j)--(l);
    \draw [-{Latex[length=2mm,width=1.5mm]}](l)--(f);

    \draw [-{Latex[length=2mm,width=1.5mm]}](d)--(g);
    \draw [-{Latex[length=2mm,width=1.5mm]}](h)--(g);

    \draw [-{Latex[length=2mm,width=1.5mm]}](m)--(c);
    \draw [-{Latex[length=2mm,width=1.5mm]}](e)--(m);
    \draw [-{Latex[length=2mm,width=1.5mm]}](n)--(e);

    \node[fill = none]at (3.85,3.6) {\textcolor{red}{1}};
    \node[fill = none]at (5.5,3.7) {\textcolor{red}{-1}};

    \draw [loosely dashed](a2)--(b2);
    \draw (a2)--(c2);
    \draw (a2)--(d2);
    \draw[loosely dashed] (a2) --(e2);
    \draw(a2) --(f2);
    \draw (a2) --(g2);
    \draw[loosely dashed] (a2)--(h2);
    \draw (a2) --(i2);

    \draw [loosely dashed](b2) --(i2);
    \draw (j2)--(k2);
    \draw (b2)--(k2);
    \draw(j2)--(l2);
    \draw (f2)--(l2);

    \draw [loosely dashed](d2)--(g2);
    \draw (g2)--(h2);

    \draw (c2)--(m2);
    \draw (e2)--(m2);
    \draw (e2)--(n2);

     \draw[loosely dashed] (a3) --(d3);
    \draw [-{Latex[length=2mm,width=1.5mm]}](g3) --(a3);
    \draw[loosely dashed] (a3) --(h3);
    \draw [-{Latex[length=2mm,width=1.5mm]}](d3)--(g3);
    \draw [-{Latex[length=2mm,width=1.5mm]}](h3)--(g3);

    \node[fill = none]at (3.8,-.6) {\textcolor{red}{-1}};
    \node[fill = none]at (3.8,.8) {\textcolor{red}{1}};

    \draw (a4) --(d4);
    \draw(g4) --(a4);
    \draw[loosely dashed] (a4) --(h4);
    \draw[loosely dashed] (d4)--(g4);
    \draw (h4)--(g4);

\end{tikzpicture}
\caption{A demonstration of Lemma \ref{dif1}.}\label{bigpic} 
\end{center}
\end{figure}
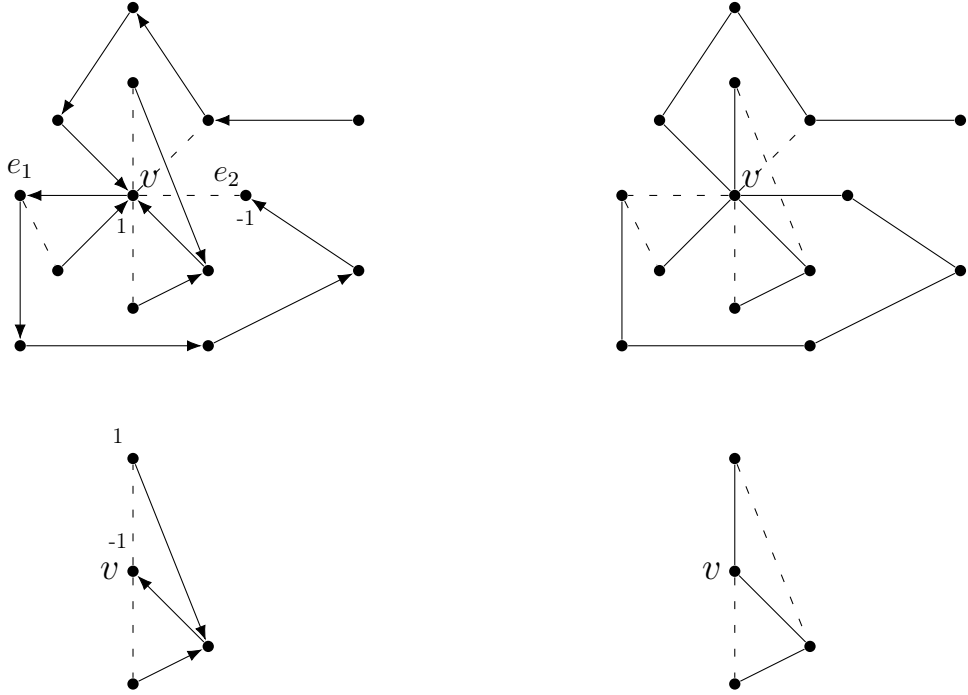

Let $\ribs$ be a ribbon graph with a vertex $v$ such that $v$ is not a cut vertex.\footnote{We use $\ribs$ instead of $\rib$ because we want to think of $\ribs$ as a v-component of a larger ribbon graph.} Let $\{e_1,...,e_n\}$ be the edges of $G'$ incident to $v$. For any $\mathcal E \subseteq \{e_1,...,e_n\}$, let $S_\mathcal E \in \text{Div}^0(G')$ be the configuration that places $-k$ chips on $v$ and $d_x$ chips on each other vertex $x \in V(G')$ where $d_x$ is the number of edges incident to $x$ in $\mathcal E$.

\begin{lem}\label{dif2}
In the construction above, if $S_\mathcal E = S_\mathcal {E'}$ then either $\mathcal E =\mathcal E'$ or one is $\{e_1,...,e_n\}$ and the other is $\emptyset$.
\end{lem}

\begin{proof} $S_\mathcal E = S_\mathcal {E'} $ if and only if $S_\mathcal E - S_\mathcal {E'} = \text{Id}$. Because $S_\mathcal E - S_\mathcal {E'}$ and $S_{\mathcal E'} - S_\mathcal E$ sum to the all zeros configuration, at least one of them must have a nonnegative number of chips placed on $v$. Call this configuration $S'$. By Lemma \ref{lem1}, if $S'$ is equivalent to the identity, then we can get from $S'$ to the all zeros configuration by firing vertices. Consider a sequence of firings that results in the all zeros configuration. Because the only way for a vertex to lose chips is to be fired, if $v$ starts with a positive number of chips, it must be fired. Furthermore, if $v$ starts with 0 chips, then unless $S'$ is already the all zeros configuration (which only occurs when $\mathcal E = \mathcal E'$), some vertex must have a positive number of chips, and must therefore fire. By definition of $S_\mathcal E$ and $S_{\mathcal E'}$, the only vertices with a possibly nonzero number of chips in $S'$ are those adjacent to $v$, so some vertex adjacent to $v$ must fire. This deposits at least one chip to $v$, which means that $v$ now has a finite number of chips and must fire as well.\\

We have shown that if $\mathcal E \not= \mathcal E'$, then $v$ must fire at some point. Since the ordering of firings is irrelevant, we can assume that $v$ fires first. Note that every vertex $x$ in $S'$ has either $0,d_x,$ or $-d_x$ chips on it where $d_x$ is the number of edges connecting $v$ to $x$. Thus, after firing $v$, each edge has either $0,d_x,$ or $2d_x$ chips. If every vertex ends up with 0 chips, then we have reached the all zeros configuration with only a single firing. This only occurs if every vertex began with $-d_x$ chips. By construction, we see that this occurs if $E = \{e_1,...,e_n\}$ and $E' = \emptyset$ (or vice versa). Otherwise, some vertex has a positive number of chips and every other vertex has a nonnegative number of chips. By the same reasoning used in the previous proposition, since we have only a single $v$-component, all non $v$ vertices must fire at least once. However, since $v$ also must fire, this means that every vertex must fire at least once when going from $S'$ to the all zeros configuration. This cannot be required since firing every vertex is equivalent to firing no vertices.\end{proof}

By combining the results of the last two lemmas, for a ribbon graph $\rib$ we are able to find exactly which edges from one $v$-component $\ribs$ fall between two sequential edges in a second $v$-component$\ribss$ with one exception. If all of the edges of $\ribs$ fall between the same two edges of $\ribss$, then we cannot always determine which pair of edges they fall between. However, the following lemma shows that any ambiguities in $\rho_v$ can be resolved with no effect on the genus of $\rib$. 

Let $\rib$ be a ribbon graph, and $v \in V(G)$ such that $\rho_v = (e_1,...,e_{i+j})$. Assume that for all $1\le k \le i$ and $i+1 \le l \le i+j$, $e_k$ and $e_l$ are on different $v$-components of $\rib$. Let $\ribs$ be the union of all $v$-components non-trivially intersecting $\{e_1,..,e_i\}$ and $\ribss$ be the union of all $v$-components non-trivially intersecting $\{e_{i+1},..,e_{i+j}\}$ where $\{\rho'_{v_k}\}$ and $\{\rho''_{v_k}\}$ are defined naturally as restrictions of $\{\rho_{v_k}\}$ (see Figure \ref{sub}).

\begin{figure}
\begin{center}
\begin{tikzpicture}

    \tikzstyle{every node} = [circle,fill,inner sep=1pt,minimum size = 1.5mm]
    \node [label = {north:{\large$v$}}](a) at (4,4){};

    \tikzstyle{every node} = [draw = none,fill = none,scale = .7]
    \draw (a)  .. controls (3,3.5) .. (2,3.2)
        node[pos = .9,below]{$e_1$};
    \draw (a)  .. controls (3,3.75) .. (2,3.6)
        node[pos = .9,below]{$e_2$};
    \draw [loosely dashed](a) .. controls (3,4) .. (2,4);
    \draw (a) .. controls (3,4.25) .. (2,4.4)
         node[pos = .9,above]{$e_{i-1}$};
    \draw (a) .. controls (3,4.5) .. (2,4.8)
         node[pos = .9,above]{$e_{i}$};

    \draw[dashed] (1.2,4) ellipse (3.2 and 1.5);

    \draw (a)  .. controls (5,3.5) .. (6,3.2)
        node[pos = .9,below]{$e_{i+j}$};
    \draw (a)  .. controls (5,3.75) .. (6,3.6)
        node[pos = .9,below]{$e_{i+j-1}$};
    \draw [loosely dashed](a) .. controls (5,4) .. (6,4);
    \draw (a) .. controls (5,4.25) .. (6,4.4)
         node[pos = .9,above]{$e_{i+2}$};
    \draw (a) .. controls (5,4.5) .. (6,4.8)
         node[pos = .9,above]{$e_{i+1}$};

    \draw[dashed] (6.8,4) ellipse (3.2 and 1.5);
    
    \tikzstyle{every node} = [draw = none,fill = none, scale = 1.5]
    \node (g) at (1.2,4) {$G'$};
    \node(h) at (6.8,4) {$G''$};
    \node(i) at (4,5.7) {{\Large$G$}};
\end{tikzpicture}
\caption{The genus of the full ribbon graph is the sum of the genera of the two ribbon subgraphs } 
\label{sub} 
\end{center}
\end{figure}
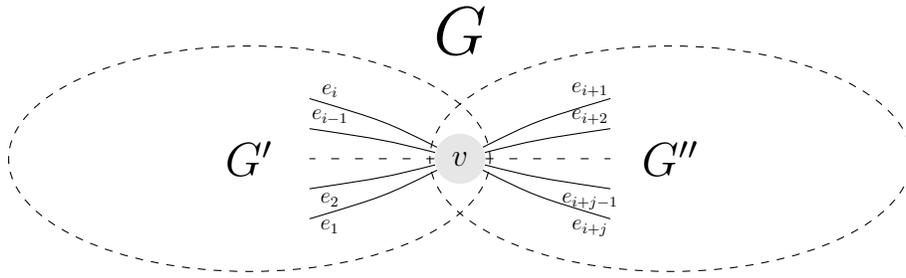

\begin{lem}\label{subgraphs}
In the above construction, the genus of $\rib$ is the sum of the genus of $\ribs$ and the genus of $\ribss$.\\ 
\end{lem}

\begin{proof} First, we note that $|E(G')| + |E(G'')| = |E(G)|$ because every edge in $G$ is in exactly one of the subgraphs. Furthermore, $|V(G')| + |V(G'')| = |V(G)| + 1$ because every vertex in $G$ is in exactly one of the subgraphs, except for $v$ which is in both.\\  

Next note that every cycle of $\rib$ is entirely contained in either $\ribs$ or $\ribss$ unless it enters $v$ on the edge $e_i$ or the edge $e_{i+j}$. We also know that these two half edges must be part of the same cycle because after the cycle leaves $\ribs$ via $e_i$, it must enter $\ribs$ again via $e_{i+j}$ or it would not be a closed loop. Because this cycle remains a cycle when restricted to either $\ribs$ or $\ribss$ it is double counted when summing $|\text{cyc}(G')|$ and $|\text{cyc}(G'')|$. Thus, we have $|\text{cyc}(G')| + |\text{cyc}(G'')| = |\text{cyc}(G)| + 1$.

Now, we use the genus formula given in Proposition \ref{kauf}, 

$$g(\rib) = \frac{|E(G)|-|V(G)|-|\text{cyc}(G)|+2}2 = $$

$$=\frac{|E(G')| + |E(G'')|-|V(G')|-|V(G'')| + 1 - |\text{cyc}(G')| - |\text{cyc}(G'')| + 1+2}2 = $$

$$=\frac{(|E(G')| -V(G') - |\text{cyc}(G_1)| +2) + (|E(G'')| - |V(G'')| - |\text{cyc}(G'')|+2)}2 =$$

$$=g(\ribs) + g(\ribss).$$ \end{proof}

We are now ready to prove Theorem \ref{algorithm}.

\begin{proof}[Proof of Theorem \ref{algorithm}]
Choose any vertex $v \in V(G)$ and consider its $v$-components. Lemma \ref{same} gives us the order of edges for each $v$-component while Lemma \ref{dif1} gives which edges of one $v$-component are between two given edges of another $v$-component. Lemma \ref{dif2} tells us that there is potential ambiguity if all of the edges in one $v$-component fall between the same two edges of another $v$-component. However, Lemma \ref{subgraphs} resolves this ambiguity by allowing us to choose arbitrarily when we cannot deduce cyclic order from the previous lemmas with no effect on the ribbon graph's genus. If we repeat this procedure for every vertex of $G$, we have deduced the cyclic orders for a ribbon graph with the same genus as $\rib$. Thus, we can use Proposition \ref{kauf} to determine the genus of $\rib$. \end{proof}

Finally, we conjecture that the same theorem holds for the Bernardi process.\\
\begin{conj}
Let $\rib$ be a ribbon graph. Suppose that we are given $V(G)$, $E(G)$, $\text{Pic}^0(G)$, $\Tau(G) \subset E(G)$ and for every $v \in V(G)$, we are given the map 

$$\text{Pic}^0(G) \times \Tau(G) \xto{\beta_v(\text{Pic}^0(G))} \Tau(G)$$

\noindent where $\beta_v$ is the Bernardi process with basepoint $v$. Then, it is possible to determine the genus of $\rib$.\\
\end{conj}

The challenge for this conjecture is that even on a cut-free graph, it is not easy to use the Bernardi process to detect information about the cyclic order around a fixed vertex without information about the cyclic order around other vertices. In other words, there is no clear analogue to Proposition \ref{same}.\\
\acknowledgements{The author would like to thank Melody Chan for asking the question that inspired this paper and Melody Chan, Caroline Klivans, Brian Freidin, Dori Bejleri, and the anonymous reviewers for a great deal of excellent revisions.}

\bibliographystyle{abbrvnat}
\bibliography{Sandpile_Torsor.bib}

\end{document}